\documentclass[oneside,english]{amsart}
\usepackage[T1]{fontenc}
\usepackage[latin9]{inputenc}
\usepackage{geometry}
\usepackage{amsthm}
\usepackage{amstext}
\usepackage{amssymb}
\usepackage{esint}

\makeatletter
\numberwithin{equation}{section}
\numberwithin{figure}{section}

\newtheorem{theorem}{Theorem}
\newtheorem{fact}[theorem]{Fact}
\newtheorem{rem}[theorem]{Remark}
\newtheorem{lem}[theorem]{Lemma}
\newtheorem{thm}[theorem]{Theorem}

\newtheorem{defi}{Definition}

\usepackage{babel}

\author{R. M. \L OCHOWSKI} 

\makeatother

\begin{document}

\title{Pathwise stochatic integration\\
 with finite variation processes\\
 uniformly approximating càdlàg processes}

\maketitle
\textbf{Abstract.} For any real-valued stochastic process $X$ with
càdlàg paths we define non-empty family of processes which have locally
finite total variation, have jumps of the same order as the process
$X$ and uniformly approximate its paths on compacts. The application
of the defined class is the definition of stochastic integral with
semimartingale integrand and integrator as a limit of pathwise Lebesgue-Stieltjes
integrals. This construction leads to the stochastic integral with
some correction term (different from the Stratonovich integral). We
compare the obtained result with classical results of Wong-Zakai and
Bichteler on pathwise stochastic integration. As a "byproduct"
we obtain an example of a series of double Skorohod maps of a standard
Brownian motion, which is not a semimartingale.

\section{Introduction}

Let $X=\left(X_{t}\right)_{t\geq0}$ be a real-valued stochastic process
with càdlàg paths and let $T\geq0.$ The total variation of the process
$X$ on the interval $\left[0;T\right]$ is defined with the following
formula 
\[
TV\left(X,T\right):=\sup_{n}\sup_{0\leq t_{0}<t_{1}<...<t_{n}\leq T}\sum_{i=1}^{n}\left|X_{t_{i}}-X_{t_{i-1}}\right|.
\]
Unfortunately, many of the most important families of stochastic processes
are characterized with a ''wild'' behavior, demonstrated by their
infinite total variation. This fact arguably caused the need of the
development of the general theory of stochastic integration. The main
idea allowing to overcome the problematic infinite total variation
and define stochastic integral with respect to a semimartingale utilizes
the fact that the quadratic variation of the semimartingale is still
finite. The similar idea may be applied when $p-$variation of the
integrator is finite for some $p\in\left(1;2\right)$. This approach
utilizes Love-Young inequality and may be used e.g. to define stochastic
integral with respect to fractional Brownian motion (cf. \cite{Kubilius:2008}).
Further developments, where Hölder continuity plays crucial role,
led to the rough paths theory developed by T. Lyons and his co-workers
(cf. \cite{Friz:2010fk}); some other generalization introduces Orlicz
norms and may be found in the recent book by Dudley and Norvaiša \cite[Chapt. 3]{DudleyNorvaisa:2011}).
The approach used in this article is somewhat different. It is similar
to the old approach of Wong and Zakai \cite{WongZakai:1965a} and
is based on the simple observation that in the neighborhood (in sup
norm) of every càdlàg function defined on compact interval $\left[0;T\right]$
one easily finds another function with finite total variation. Thus,
for every $c>0,$ the process $X$ may be decomposed as the sum 
\[
X=X^{c}+\left(X-X^{c}\right)
\]
where $X^{c}$ is a ``nice'' process with finite total variation
and the difference $X-X^{c}$ is a process with small amplitude (no
greater than $K_{T}c$) but possibly ``wild'' behaviour with infinite
total variation. More precisely, let $F$ be some fixed, right continuous
filtration such that $X$ is adapted to $F.$ Now, for every $c>0$
we introduce (non-empty, as it will be shown in the sequel) family
$\mathcal{X}^{c}$ of processes with càdlàg paths, satisfying the
following conditions. If $X^{c}\in\mathcal{X}^{c}$ then 
\begin{enumerate}
\item the process $X^{c}$ has locally finite total variation; 
\item $X^{c}$ has càdlàg paths; 
\item for every $T\geq0$ there exists such $K_{T}<+\infty$ that for every
$t\in\left[0;T\right],$ $\left|X_{t}-X_{t}^{c}\right|\leq K_{T}c;$ 
\item for every $T\geq0$ there exists such $L_{T}<+\infty$ that for every
$t\in\left[0;T\right],$ $\left|\Delta X_{t}^{c}\right|\leq L_{T}\left|\Delta X_{t}\right|;$ 
\item the process $X^{c}$ is adapted to the filtration $F.$ 
\end{enumerate}
We will prove that if processes $X$ and $Y$ are càdlàg semimartingales
on a filtered probability space $\left(\Omega,\left(\mathcal{F}_{t}\right)_{t\geq0},\mathbb{P},F\right),$
with a probability measure $\mathbb{P},$ such that usual hypotheses
hold (cf. \cite[Sect. 1.1]{Protter:2004}), then the sequence of pathwise
Lebesgue-Stieltjes integrals 
\begin{equation}
\int_{0}^{T}Y_{-}\mathrm{d}X^{c},\mbox{ }c>0,\label{eq:loch}
\end{equation}
with $X^{c}\in{\mathcal{X}}^{c}$, tends uniformly in probability
$\mathbb{P}$ on compacts to $\int_{0}^{T}Y_{-}\mathrm{d}X+\left[X^{cont},Y^{cont}\right]_{T}$ as $c \downarrow 0$;
$\int_{0}^{T}Y_{-}\mathrm{d}X$ denotes here the (semimartingale)
stochastic integral and $X^{cont}$ and $Y^{cont}$ denote continuous
parts of $X$ and $Y$ respectively. Moreover, for any square summable
sequence $\left(c\left(n\right)\right)_{n\geq1}$ we get a.s., uniform
on compacts, convergence of the sequence $\int_{0}^{T}Y_{-}\mathrm{d}X^{c\left(n\right)},n=1,2,...$
(cf. Theorem \ref{thm:leb_stieltjes_stoch-1}).

We shall stress here that for each $c>0$ and each pair of càdlàg
paths $\left(X\left(\omega\right),Y\left(\omega\right)\right),\omega\in\Omega,$
the value of $\int_{0}^{T}Y_{-}\left(\omega\right)\mathrm{d}X^{c}\left(\omega\right)$
(and thus the limit, if it exists) is independent of the probability
measure $\mathbb{P}$. Thus we obtain a result in the spirit of Wong
and Zakai \cite{WongZakai:1965a}, Bichtelier (see \cite{Bichteler:1981}
or \cite{Karandikar:1995}) or the recent result of Nutz \cite{Nutz:2011},
where operations almost surely leading to the stochastic integral,
independent of probability measures and filtrations, are considered.
The old approach of Wong and Zakai is very straightforward, since
it just replaces stochastic integral with Lebesgue-Stieltjes integral.
However, it deals with very limited family of possible integrands
and integrators (diffusions driven by a Brownian motion), $x_t = \int_0^t g(s)\mathrm{d}s +  \int_{0}^{t}f(s)\mathrm{d}B_{s},$
and using appropriate continuous, finite variation approximation of
$x,$ $x^{n},$ one gets a.s. in the limit the Stratonovich integral
\[
\lim_{n\rightarrow\infty}\int_{0}^{T}\psi\left(x_{t}^{n},t\right)\mathrm{d}x_{t}^{n}=\int_{0}^{T}\psi\left(x_{t},t\right)\mathrm{d}x_{t}+\frac{1}{2}\int_{0}^{T}f^2(t)\frac{\partial\psi}{\partial x}\left(x_{t},t\right)\mathrm{d}t.
\]
(Modification of this approach is possible, cf. \cite{JMP1989}, but
it requires introduction of a probability measure on the Skorohod
space and rather strong UT (uniform tightness) condition, which is
sometimes difficult to verify in practice, and is not satisfied e.g.
by the picewise linear approximation of the Wiener process. Moreover,
the obtained convergence holds in distribution.)

Bichteler's remarkable approach allows to integrate any adapted càdlàg
process $Y$ with semimartingale integrator $X,$ and is based on
the approximation 
\begin{equation}
\lim_{n\rightarrow\infty}\sup_{0\leq t\leq T}\left|Y_{0}X_{0}+\sum_{i=1}^{\infty}Y_{\tau_{i-1}^{n}\wedge t}\left(X_{\tau_{i}^{n}\wedge t}-X_{\tau_{i-1}^{n}\wedge t}\right)-\int_{0}^{t}Y_{-}\mathrm{d}X\right|=0\mbox{ a.s.},\label{eq:bicht}
\end{equation}
where $\tau^{n}=\left(\tau_{i}^{n}\right),$ $i=0,1,2,\ldots,$ is
the following sequence of stopping times: $\tau_{0}^{n}=0$ and for
$i=1,2,\ldots,$ 
\[
\tau_{i}^{n}=\inf\left\{ t>\tau_{i-1}^{n}:\left|Y_{t}-Y_{\tau_{i-1}^{n}}\right|\geq2^{-n}\right\} .
\]
\begin{rem} Following the proof of \cite[Theorem 2]{Karandikar:1995}
it is easy to see that Bichteler's construction works for any sequence
$\tau^{n}=\left(\tau_{i}^{n}\right),$ $i=0,1,2,\ldots,$ of stopping
times, such that $\tau_{0}^{n}=0$ and 
\[
\tau_{i}^{n}=\inf\left\{ t>\tau_{i-1}^{n}:\left|Y_{t}-Y_{\tau_{i-1}^{n}}\right|\geq c\left(n\right)\right\} ,
\]
for $i=1,2,\ldots,$ given $c(n)>0,$ $\sum_{n=1}^{\infty}c^{2}\left(n\right)<+\infty.$\end{rem}

The new result of Nutz goes even further, since it does not assume
càdlàg property of the integrand, but to prove his result one needs
the existence of Mokobodzki's medial limits (cf. \cite{Meyer:1973}),
which one can not prove under standard Zermelo\textendash{}Fraenkel
set theory with the axiom of choice.

The results of this paper seem to indicate that Bichteler's approach
is the most flexible (under standard Zermelo\textendash{}Fraenkel
set theory with the axiom of choice) since we will prove that even
in the case when the integrator is a standard Brownian motion, our
construction (\ref{eq:loch}) can not be extended to arbitrary adapted, 
bounded by a constant, continuous integrand $Y$. Moreover, similar to the Wong
- Zakai, but a more general construction, 
\[
\int_{0}^{T}Z_{-}^{c}\mathrm{d}X^{c},\mbox{ }c>0,
\]
can not be extended to arbitrary continuous semimartingale integrand
$Z$ and semimartingale integrator $X.$ The construction of the appropriate
$Y$ and $Z,$ adapted to the natural filtration of $B$ and leading
to divergent series of integrals $\int_{0}^{T}Y\mathrm{d}B^{\gamma\left(n\right)},$
$\int_{0}^{T}Z^{\delta\left(n\right)}\mathrm{d}\tilde{B}^{\delta\left(n\right)},$
with $B^{\gamma\left(n\right)},$ $Z^{\delta\left(n\right)},$ $\tilde{B}^{\delta\left(n\right)}$
satisfying conditions (1)-(5) for some semimartingales $Z,$ $\tilde{B}$,
with $\gamma\left(n\right),$ $\delta\left(n\right)\downarrow0$ as
$n\uparrow+\infty,$ will utilise the recent findings of Bednorz, \L{}ochowski
and Mi\l{}o\'{s} on truncated variation (see \cite{BL}, \cite{LM2011}) and its
relation with the double Skorohod map on $[-c;c]$ (cf. \cite{BKR}).

Let us shortly comment on the organization of the paper. In the next
section we prove, for any $c>0,$ the existence of non-empty family
of processes ${\mathcal{X}}^{c}$. In the third section we deal with
the limit of pathwise, Lebesgue-Stieltjes integrals $\int_{0}^{T}Y_{-}\mathrm{d}X^{c}$
as $c\downarrow0.$ The fourth section is devoted to the construction
of counterexamples. Last section - Appendix - summarizes the necessary
facts on the relation between the truncated variation and double Skorohod
map on $[-c;c].$

\textsl{Acknowledgements.} The author would like to thank Dr. Alexander
Cox for pointing out to him the results of \cite{Nutz:2011}. Research
of the author was supported by the African Institute for Mathematical
Sciences in Muizenberg, South Africa, and by the National Science
Centre in Poland under decision no. UMO-2011/01/B/ST1/05089.

\section{Existence of the sequence $\left(X^{c}\right)_{c>0}$\label{sec:Existence}}

In this section we will prove that for every $c>0$ the family of
processes ${\mathcal{X}}^{c},$ satisfying the conditions (1)-(5)
of Section 1 is non-empty. For given $c>0$ we will simply construct
a process $X^{c}$ satisfying all these conditions. We start with
few definitions.

For fixed $c>0$ we define two stopping times 
\begin{gather*}
T_{u}^{2c}X=\inf\left\{ s\geq0:\mbox{ }\sup_{t\in\left[0;s\right]}X_{t}-X_{0}>c\right\} ,\\
T_{d}^{2c}X=\inf\left\{ s\geq0:\mbox{ }X_{0}-\inf_{t\in\left[0;s\right]}X_{t}>c\right\} .
\end{gather*}
Assume that $T_{d}^{2c}X\geq T_{u}^{2c}X,$ i.e. the first upward
jump of the process $X$ from $X_{0}$ of size $c$ appears before
the first downward jump of the same size $c$ or both times are infinite
(there is no upward or downward jump of size $c$). Note that in the
case $T_{d}^{2c}X<T_{u}^{2c}X$ we may simply consider the process
$-X.$ Now we define sequences $\left(T_{d,k}^{2c}\right)_{k=1}^{\infty},\mbox{ }\left(T_{u,k}^{2c}\right)_{k=1}^{\infty}$
in the following way: $T_{u,0}^{2c}=T_{u}^{2c}X$ and for $k=0,1,2,...$
\begin{gather*}
T_{d,k}^{2c}=\left\{ \begin{array}{lr}
\inf\left\{ s\geq T_{u,k}^{2c}:\sup_{t\in\left[T_{u,k}^{2c};s\right]}X_{t}-X_{s}>2c\right\}  & \text{ if }T_{u,k}^{2c}<+\infty,\\
+\infty & \mbox{otherwise},
\end{array}\right.\\
T_{u,k+1}^{2c}=\left\{ \begin{array}{lr}
\inf\left\{ s\geq T_{d,k}^{2c}:X_{s}-\inf_{t\in\left[T_{d,k}^{2c};s\right]}X_{t}>2c\right\}  & \text{ if }T_{d,k}^{2c}<+\infty,\\
+\infty & \mbox{otherwise}.
\end{array}\right.
\end{gather*}
\begin{rem} \label{finK} Note that for any $s>0$ there exists such
$K<\infty$ that $T_{u,K}^{2c}>s$ or $T_{d,K}^{2c}>s.$ Otherwise
we would obtain two infinite sequences $\left(s_{k}\right)_{k=1}^{\infty},\left(S_{k}\right)_{k=1}^{\infty}$
such that $0\leq s\left(1\right)<S\left(1\right)<s\left(2\right)<S\left(2\right)<...\leq s$
and $X_{S\left(k\right)}-X_{s\left(k\right)}\geq c.$ But this is
a contradiction since $X$ is a càdlàg process and for any sequence
such that $0\leq s\left(1\right)<S\left(1\right)<s\left(2\right)<S\left(2\right)<...\leq s$
sequences $\left(X_{S\left(k\right)}\right)_{k=1}^{\infty},\left(X_{s\left(k\right)}\right)_{k=1}^{\infty}$
have a common limit. \end{rem}

Now we define, for the given process $X,$ the process $X^{c}$ with
the formulas 
\begin{equation}
X_{s}^{c}=\left\{ \begin{array}{lr}
X_{0} & \text{ if }s\in\left[0;T_{u,0}^{2c}\right);\\
\sup_{t\in\left[T_{u,k}^{2c};s\right]}X_{t}-c & \text{ if }s\in\left[T_{u,k}^{2c};T_{d,k}^{2c}\right),k=0,1,2,...;\\
\inf_{t\in\left[T_{d,k}^{2c};s\right]}X_{t}+c & \text{ if }s\in\left[T_{d,k}^{2c};T_{u,k+1}^{2c}\right),k=0,1,2,....
\end{array}\right.\label{eq:defXc}
\end{equation}

\begin{rem} Note that due to Remark \ref{finK}, $s$ belongs to
one of the intervals $\left[0;T_{u,0}^{2c}\right),\left[T_{u,k}^{2c};T_{d,k}^{2c}\right)$
or $\left[T_{d,k}^{2c};T_{u,k+1}^{c}\right)$ for some $k=0,1,2,...$
and the process $X_{s}^{c}$ is defined for every $s\geq0.$ \end{rem}

Now we are to prove that $X^{c}$ satisfies conditions (1)-(5).

\begin{proof}

(1) The process $X^{c}$ has finite total on compact intervals, since
it is monotonic on intervals of the form $\left[T_{u,k}^{2c};T_{d,k}^{2c}\right),$
$\left[T_{d,k}^{2c};T_{u,k+1}^{c}\right)$ which sum up to the whole
half-line $\left[0;+\infty\right).$

(2) From formula (\ref{eq:defXc}) it follows that $X^{c}$ is also
càdlàg.

(3) In order to prove condition (3) we consider 3 possibilities. 
\begin{itemize}
\item $s\in\left[0;T_{u,0}^{2c}\right).$ In this case, since $0\leq s<T_{u}^{2c}X\leq T_{d}^{2c}X,$
by definition of $T_{u}^{2c}X$ and $T_{d}^{2c}X,$ 
\[
X_{s}-X_{s}^{c}=X_{s}-X_{0}\in\left[-c;c\right].
\]

\item $s\in\left[T_{u,k}^{2c};T_{d,k}^{2c}\right),$ for some $k=0,1,2,...$
In this case, by definition of $T_{d,k}^{2c}$, $\sup_{t\in\left[T_{u,k}^{2c};s\right]}X_{t}-X_{s}$
belongs to the interval $\left[0;2c\right],$ hence 
\[
X_{s}-X_{s}^{c}=X_{s}-\sup_{t\in\left[T_{u,k}^{2c};s\right]}X_{t}+c\in\left[-c;c\right].
\]

\item $s\in\left[T_{d,k}^{2c};T_{u,k+1}^{2c}\right)$ for some $k=0,1,2,...$
In this case $X_{s}-\inf_{t\in\left[T_{d,k}^{2c};s\right]}X_{t}$
belongs to the interval $\left[0;2c\right],$ hence 
\[
X_{s}-X_{s}^{c}=X_{s}-\inf_{t\in\left[T_{d,k}^{2c};s\right]}X_{t}-c\in\left[-c;c\right].
\]

\end{itemize}
(4) We will prove stronger fact than (4), namely that for every $s>0,$
\begin{align}
\left|\Delta X_{s}^{c}\right| & \leq\left|\Delta X_{s}\right|.\label{eq:jumps}
\end{align}
Indeed, from formula (\ref{eq:defXc}) it follows that for any $s\notin\left\{ T_{u,k}^{2c};T_{d,k}^{2c}\right\} ,$
(\ref{eq:jumps}) holds, hence let us assume that $s\in\left\{ T_{u,k}^{2c};T_{d,k}^{2c}\right\} .$
We consider several possibilities. If $s=T_{u,0}^{2c}$ then, by the
definition of $T_{u,0}^{2c},$ 
\[
X_{s}^{c}-X_{s-}^{c}=X_{s}-c-X_{0}\mbox{ \ensuremath{\geq}0 and }X_{s}^{c}-X_{s-}^{c}=X_{s}-X_{0}-c\leq X_{s}-X_{s-}.
\]
If $s=T_{u,k}^{2c},k=1,2,...,$ then, by the definition of $T_{u,k}^{2c},$
\[
X_{s}^{c}-X_{s-}^{c}=X_{s}-c-\left(\inf_{t\in\left[T_{d,k-1}^{2c};s\right]}X_{t}+c\right)\mbox{ =\ensuremath{X_{s}-\inf_{t\in\left[T_{d,k-1}^{2c};s\right]}X_{t}-2c\geq0}}
\]
and, on the other hand, 
\[
X_{s}^{c}-X_{s-}^{c}\mbox{ =\ensuremath{X_{s}-\inf_{t\in\left[T_{d,k-1}^{2c};s\right]}X_{t}-2c\leq X_{s}-X_{s-}.}}
\]
Similar arguments may be applied for $s=T_{d,k}^{2c},k=0,1,....$

(5) The process $X^{c}$ is adapted to the filtration $F$ since it
is adapted to any right continuous filtration containing the natural
filtration of the process $X$.

\end{proof}

\begin{rem} \label{Skoro} It is possible to define the process $X^{c}$
in many different ways. For example, defining 
\[
X^{c}=X_{0}+UTV^{c}\left(X,\cdot\right)-DTV^{c}\left(X,\cdot\right)
\]
we obtain a process satisfying all conditions (1)-(5) and having (on
the intervals of the form $\left[0;T\right],\mbox{ }T>0$) the smallest
possible total variation among all processes, increments of which
differ from the increments of the process $X$ by no more than $c.$
$UTV^{c}\left(X,\cdot\right)\mbox{ and }DTV^{c}\left(X,\cdot\right)$
denote here upward and downward truncated variation processes, defined
as 
\begin{eqnarray*}
UTV^{c}\left(X,t\right) & := & \sup_{n}\sup_{0\leq t_{1}<t_{2}<...<t_{n}\leq t}\sum_{i=1}^{n}\max\left\{ X_{t_{i}}-X_{t_{i-1}}-c,0\right\} ,\\
DTV^{c}\left(X,t\right) & := & \sup_{n}\sup_{0\leq t_{1}<t_{2}<...<t_{n}\leq t}\sum_{i=1}^{n}\max\left\{ X_{t_{i-1}}-X_{t_{i}}-c,0\right\} .
\end{eqnarray*}
Moreover, for any $T>0$ we have 
\begin{eqnarray*}
TV\left(X^{c};T\right) & = & UTV^{c}\left(X,T\right)+DTV^{c}\left(X,T\right)\\
 & = & \sup_{n}\sup_{0\leq t_{1}<t_{2}<...<t_{n}\leq T}\sum_{i=1}^{n}\max\left\{ \left|X_{t_{i}}-X_{t_{i-1}}\right|-c,0\right\} =:TV^{c}\left(X;T\right).
\end{eqnarray*}
We will call $TV^{c}$ \emph{truncated variation}. For more on truncated
variation, upward truncated variation and downward truncated variation
see \cite{L2012QM} or \cite{LM2011}.

Some other construction may be done with the Skorohod map on $\left[-\alpha^{c};\beta^{c}\right]$
(cf. \cite{BKR}) where $\alpha^{c},\beta^{c}:\left[0;+\infty\right)\rightarrow\left(0;+\infty\right)$
are (possibly time-dependent) continuous boundaries such that $\sup_{0\leq t\leq T}\alpha^{c}\left(t\right)\leq K_{T}c,$
$\sup_{0\leq t\leq T}\beta^{c}\left(t\right)\leq K_{T}c$ and $\inf_{0\leq t}\left(\beta^{c}\left(t\right)+\alpha^{c}\left(t\right)\right)>0.$
The Skorohod map on $\left[-\alpha^{c};\beta^{c}\right]$ allows to
construct such a locally finite variation càdlàg process $-X^{c}$
that 
\[
X+\left(-X^{c}\right)\in\left[-\alpha^{c};\beta^{c}\right].
\]
From further properties of this map it follows that $X^{c}$ satisfies
all conditions (1)-(5). In fact, construction (\ref{eq:defXc}) of
$X^{c}$ is based on a Skorohod map on the interval $\left[-c;c\right].$
In the Appendix we will prove this as well as other interesting properties
of this map. \end{rem}

\section{Pathwise Lebesgue-Stieltjes integration with respect to the process
$X^{c}$}

Let us now consider a measurable space $\left(\Omega,{\mathcal{F}}\right)$
equipped with a right-continuous filtration $F$ and two processes
$X$ and $Y$ with càdlàg paths, adapted to $F.$ For $T>0$ and for
a sequence of processes $\left(X^{c}\right)_{c>0}$ with $X^{c}\in{\mathcal{X}}^{c}$
let us consider the sequence 
\begin{equation}
\int_{0}^{T}Y_{-}\mathrm{d}X^{c}.\label{eq:int_c}
\end{equation}
The integral in (\ref{eq:int_c}) is understood in the pathwise, Lebesgue-Stieltjes
sense (recall that for any $c>0,$ $X^{c}$ has bounded variation).
We have

\begin{thm} \label{thm:leb_stieltjes_stoch} Assume that $\mathbb{P}$
is a probability measure on $\left(\Omega,{\mathcal{F}}\right)$ such
that $X$ and $Y$ are semimartingales with respect to this measure
and filtration $F,$ which is complete under $\mathbb{P}$, then 
\[
\int_{0}^{T}Y_{-}\mathrm{d}X^{c}\rightarrow^{ucp\mathbb{P}}\int_{0}^{T}Y_{-}\mathrm{d}X+\left[X^{cont},Y^{cont}\right]_{T}\,\mbox{as }c\downarrow0,
\]
where ``$\rightarrow^{ucp\mathbb{P}}$'' denotes uniform convergence
on compacts in probability $\mathbb{P}$ and $\left[X^{cont},Y^{cont}\right]_{T}$
denotes quadratic covariation of continuous parts $X^{cont},Y^{cont}$
of $X$ and $Y$ respectively.\end{thm}

\begin{proof} Fixing $c>0$ and using integration by parts formula
(cf. \cite[formula (1), page 519]{Kallenberg:2002}) we get 
\[
Y_{T}X_{T}^{c}-Y_{0}X_{0}^{c}=\int_{0}^{T}Y_{t-}\mathrm{d}X_{t}^{c}+\int_{0}^{T}X_{t-}^{c}\mathrm{d}Y_{t}+\left[Y,X^{c}\right]_{T}
\]
(the above equality and subsequent equalities in the proof hold $\mathbb{P}$
a.s.). By the uniform convergence, $X_{t}^{c}\rightrightarrows X_{t}$
as $c\downarrow0$ (note that the bound $\left|X^{c}\right|\leq\left|X\right|+K_{T}c$
and a.s. pointwise convergence $X_{t}^{c}\rightarrow X_{t}$ as $c\downarrow0$
are sufficient) we get 
\[
\int_{0}^{T}X_{t-}^{c}\mathrm{d}Y_{t}\rightarrow^{ucp\mathbb{P}}\int_{0}^{T}X_{t-}\mathrm{d}Y_{t}.
\]
Since $X^{c}$ has locally finite variation, we have (cf. \cite[Theorem 26.6 (viii)]{Kallenberg:2002}),
\[
\left[Y,X^{c}\right]_{T}=\sum_{0<s\leq T}\Delta Y_{s}\Delta X_{s}^{c}.
\]
We calculate the (pathwise) limit 
\[
\lim_{c\downarrow0}\left[Y,X^{c}\right]_{T}=\lim_{c\downarrow0}\sum_{0<s\leq T}\Delta Y_{s}\Delta X_{s}^{c}=\sum_{0<s\leq T}\Delta Y_{s}\Delta X_{s}
\]
(notice that for any $0\leq s\leq T,$ $\left|\Delta X_{s}^{c}\right|\leq L_{T}\left|\Delta X_{s}\right|,$
thus the above sum is convergent by dominated convergence) and finally
obtain 
\begin{eqnarray}
\int_{0}^{T}Y_{t-}\mathrm{d}X_{t}^{c} & = & \left\{ Y_{T}X_{T}^{c}-Y_{0}X_{0}^{c}-\int_{0}^{T}X_{t-}^{c}\mathrm{d}Y_{t}-\left[Y,X^{c}\right]_{T}\right\} \nonumber \\
 & \rightarrow^{ucp\mathbb{P}} & Y_{T}X_{T}-Y_{0}X_{0}-\int_{0}^{T}X_{t-}\mathrm{d}Y_{t}-\sum_{0<s\leq T}\Delta Y_{s}\Delta X_{s}\mbox{ as }\ensuremath{c\downarrow0.}\label{eq:lim_int_part}
\end{eqnarray}
On the other hand, again by the integration by parts formula, we obtain
\begin{equation}
\int_{0}^{T}X_{t-}\mathrm{d}Y_{t}=Y_{T}X_{T}-Y_{0}X_{0}-\int_{0}^{T}Y_{t-}\mathrm{d}X_{t}-\left[Y,X\right]_{T}.\label{eq:int_part2}
\end{equation}
Finally, comparing (\ref{eq:lim_int_part}) and (\ref{eq:int_part2}),
and using \cite[Corollary 26.15]{Kallenberg:2002}, we obtain 
\begin{eqnarray*}
\int_{0}^{T}Y_{t-}\mathrm{d}X_{t}^{c} & \rightarrow^{ucp\mathbb{P}} & \int_{0}^{T}Y_{t-}\mathrm{d}X_{t}+\left[Y,X\right]_{T}-\sum_{0<s\leq T}\Delta Y_{s}\Delta X_{s}\mbox{ \mbox{ as }\ensuremath{c\downarrow0}}\\
 & = & \int_{0}^{T}Y_{t-}\mathrm{d}X_{t}+\left[X^{cont},Y^{cont}\right]_{T}.
\end{eqnarray*}
\end{proof}

Note that to prove Theorem \ref{thm:leb_stieltjes_stoch} we did not
need the pathwise uniform convergence of the processes $X^{c}$ to
the process $X;$ we might simply use local boundedness and a.s. pointwise
convergence $X_{t}^{c}\rightarrow X_{t}$ as $c\downarrow0.$ Using
the pathwise uniform convergence of the sequence $\left(X^{c}\right)_{c>0}$
we are able to prove a bit stronger result. We have

\begin{thm} \label{thm:leb_stieltjes_stoch-1} Assume that $\mathbb{P}$
is a probability measure on $\left(\Omega,{\mathcal{F}}\right)$ such
that $X$ and $Y$ are semimartingales with respect to this measure
and filtration $F,$ which is complete under $\mathbb{P},$ then for
any $T>0$ and any sequence $\left(c\left(n\right)\right)_{n\geq1}$
such that $c(n)>0,$ $\sum_{n=1}^{\infty}c\left(n\right)^{2}<+\infty$ we have
\[
\lim_{n\rightarrow+\infty}\sup_{0\leq t\leq T}\left|\int_{0}^{t}Y_{-}\mathrm{d}X^{c\left(n\right)}-\int_{0}^{t}Y_{-}\mathrm{d}X-\left[X^{cont},Y^{cont}\right]_{t}\right|=0\mbox{ }\mathbb{P}\mbox{ a.s.}
\]

\end{thm}

\begin{proof} Using integration by parts formula and the inequality
$\left|X^{c}-X\right|\leq K_{T}c,$ we estimate 
\begin{eqnarray*}
\lefteqn{\left|\int_{0}^{t}Y_{-}\mathrm{d}X^{c}-\int_{0}^{t}Y_{-}\mathrm{d}X-\left[X^{cont},Y^{cont}\right]_{t}\right|}\\
 & = & \left|Y_{t}\left(X_{t}^{c}-X_{t}\right)-Y_{0}\left(X_{0}^{c}-X_{0}\right)-\sum_{0<s\leq t}\Delta Y_{s}\Delta\left(X_{s}^{c}-X_{s}\right)-\int_{0}^{t}\left(X_{-}^{c}-X\right)\mathrm{d}Y\right|\\
 & \leq & K_{T}c\left(\left|Y_{0}\right|+\left|Y_{t}\right|\right)+\left|\sum_{0<s\leq t}\Delta Y_{s}\Delta\left(X_{s}^{c}-X_{s}\right)\right|+\left|\int_{0}^{t}\left(X_{-}^{c}-X\right)\mathrm{d}Y\right|.
\end{eqnarray*}
Thus we get 
\begin{eqnarray*}
\lefteqn{\sup_{0\leq t\leq T}\left|\int_{0}^{t}Y_{-}\mathrm{d}X^{c}-\int_{0}^{t}Y_{-}\mathrm{d}X-\left[X^{cont},Y^{cont}\right]_{t}\right|}\\
 & \leq & K_{T}c\left(\left|Y_{0}\right|+\sup_{0\leq t\leq T}\left|Y_{t}\right|\right)+\sup_{0\leq t\leq T}\left|\sum_{0<s\leq t}\Delta Y_{s}\Delta\left(X_{s}^{c}-X_{s}\right)\right|+\sup_{0\leq t\leq T}\left|\int_{0}^{t}\left(X_{-}^{c}-X\right)\mathrm{d}Y\right|.
\end{eqnarray*}

Since $Y$ has càdlàg paths, it is locally bounded and hence $K_{T}c\left(\left|Y_{0}\right|+\sup_{0\leq t\leq T}\left|Y_{t}\right|\right)\rightarrow0$
$\mathbb{P}$ a.s. as $c\downarrow0.$

Since for every $t\in[0;T],$ $\left|X_{t}^{c}-X_{t}\right|\leq K_{T}c$
(condition (3)), for $s\in\left[0;t\right]$ we have $\left|\Delta\left(X_{s}^{c}-X_{s}\right)\right|\leq2K_{T}c.$
Similarly, by condition (4), 
\[
\left|\Delta\left(X_{s}^{c}-X_{s}\right)\right|\leq\left|\Delta X_{s}^{c}\right|+\left|\Delta X_{s}\right|\leq\left(L_{T}+1\right)\left|\Delta X_{s}\right|.
\]
Thus we obtain that 
\[
\left|\Delta\left(X_{s}^{c}-X_{s}\right)\right|\leq\min\left\{ 2K_{T}c,\left(L_{T}+1\right)\left|\Delta X_{s}\right|\right\} \leq\left(2K_{T}+L_{T}+1\right)\min\left\{ c,\left|\Delta X_{s}\right|\right\} 
\]
and using this, we estimate 
\begin{align*}
 & \sup_{0\leq t\leq T}\left|\sum_{0<s\leq t}\Delta Y_{s}\left(\Delta X_{s}^{c}-\Delta X_{s}\right)\right|\leq\sup_{0\leq t\leq T}\sqrt{\sum_{0<s\leq t}\left|\Delta Y_{s}\right|^{2}}\sqrt{\sum_{0<s\leq t}\left|\Delta\left(X_{s}^{c}-X_{s}\right)\right|^{2}}\\
 & =\sqrt{\sum_{0<s\leq T}\left|\Delta Y_{s}\right|^{2}}\sqrt{\sum_{0<s\leq T}\left|\Delta\left(X_{s}^{c}-X_{s}\right)\right|^{2}}\\
 & \leq\sqrt{\left[\Delta Y\right]_{T}}\left(2K_{T}+L_{T}+1\right)\sqrt{\sum_{0<s\leq T}\min\left\{ c^{2},\left|\Delta X_{s}\right|^{2}\right\} }\rightarrow0\mbox{ }\mathbb{P}\mbox{ a.s.}\mbox{ as }c\downarrow0.
\end{align*}

In order to estimate 
\[
I^{c}(T):=\sup_{0\leq t\leq T}\left|\int_{0}^{t}\left(X_{-}^{c}-X_{-}\right)\mathrm{d}Y\right|
\]
let us decompose the semimartingale $Y$ into a local martingale $M$
with bounded jumps (hence a local $L^{2}$ martingale) and a process
$A$ with locally finite variation (this is possible due to \cite[Lemma 26.5]{Kallenberg:2002}
but the decomposition may depend on the measure $\mathbb{P}),$ $Y=M+A.$
Let $\left(\tau\left(k\right)\right)_{k\geq1}$ be a sequence of stopping
times increasing to $+\infty$ such that $\left(M_{t\wedge\tau\left(k\right)}\right)_{t\geq0}$
is a square integrable martingale. We will use elementary estimate
$\left(a+b\right)^{2}\leq2a^{2}+2b^{2},$ the Burkholder inequality
and localization. On the set $\Omega_{N}=\left\{ \omega\in\Omega:TV\left(A,T\right)\leq N\right\} $
we have 
\begin{eqnarray*}
\lefteqn{\mathbb{E}\left[\sup_{0\leq t\leq T\wedge\tau\left(k\right)}\left|\int_{0}^{t}\left(X_{-}^{c}-X_{-}\right)\mathrm{d}Y\right|^{2};\Omega_{N}\right]}\\
 & \leq & 2\mathbb{E}\sup_{0\leq t\leq T\wedge\tau\left(k\right)}\left|\int_{0}^{t}\left(X_{-}^{c}-X_{-}\right)\mathrm{d}M\right|^{2}+2\left[\mathbb{E}\left|\int_{0}^{T}\left|X_{-}^{c}-X_{-}\right|\mathrm{d}A\right|^{2};\Omega_{N}\right]\\
 & \leq & 2\left(4K_{T}^{2}c^{2}\mathbb{E}\left[M,M\right]_{T\wedge\tau\left(k\right)}+K_{T}^{2}c^{2}N^{2}\right)\leq8\left(\mathbb{E}\left[M,M\right]_{T\wedge\tau\left(k\right)}+N^{2}\right)K_{T}^{2}c^{2}.
\end{eqnarray*}

Let now $\left(c\left(n\right)\right)_{n\geq1}$ be such a sequence
that $\sum_{n=1}^{\infty}c\left(n\right)^{2}<+\infty.$ We have 
\begin{align*}
 & \mathbb{E}\left[\sum_{n=1}^{\infty}\sup_{0\leq t\leq T\wedge\tau\left(k\right)}\left|\int_{0}^{t}\left(X_{-}^{c\left(n\right)}-X_{-}\right)\mathrm{d}Y\right|^{2};\Omega_{N}\right]\\
 & =\sum_{n=1}^{\infty}\mathbb{E}\left[\sup_{0\leq t\leq T\wedge\tau\left(k\right)}\left|\int_{0}^{t}\left(X_{-}^{c\left(n\right)}-X_{-}\right)\mathrm{d}Y\right|^{2};\Omega_{N}\right]\\
 & \leq8\left(\mathbb{E}\left[M,M\right]_{T\wedge\tau\left(k\right)}+N^{2}\right)K_{T}^{2}\sum_{n=1}^{\infty}c\left(n\right)^{2}<+\infty.
\end{align*}
Hence, the sequence $I^{c\left(n\right)}(T\wedge\tau\left(k\right)),n=1,2,...,$
converges to $0$ on the set $\Omega_{N}.$ Since $\Omega=\bigcup_{N\geq1}\Omega_{N},$
we get that $I^{c\left(n\right)}(T\wedge\tau\left(k\right))$ converges
$\mathbb{P}$ a.s. to $0.$ Finally, since $\tau\left(k\right)\rightarrow+\infty$
a.s. we get that $I^{c\left(n\right)}(T)$ converges $\mathbb{P}$
a.s. to $0.$

\end{proof}

\section{Counterexamples \label{sec:Counterexamples}}

In this section, using further properties of the sequence $X^{c}$
defined in Section \ref{sec:Existence}, which we prove in the Appendix,
we will show that even for the integrator $X=B$ being a standard
Brownian motion Theorem \ref{thm:leb_stieltjes_stoch} can not be
extended to the case when $Y$ is not a semimartingale. To prove this
we start with few definitions. First, we define sequence $\beta\left(n\right),$
$n=1,2,\ldots$ in the following way $\beta\left(1\right)=1$ and
for $n=2,3,\ldots$ 
\[
\beta\left(n\right)=n^{2}\beta\left(n-1\right)^{6}.
\]
Now we define $\alpha\left(n\right):=\beta\left(n\right)^{1/2},$
$\gamma\left(n\right):=\beta\left(n\right)^{-1}$ and 
\[
Y:=\sum_{n=2}^{\infty}\alpha\left(n\right)\left(B-B^{\gamma\left(n\right)}\right),
\]
where $B$ is a standard Brownian motion and for any $c>0,$ $B^{c}$
is defined as in Section \ref{sec:Existence} (with formulas (\ref{eq:defXc})
or symmetric). Notice that $Y$ is well defined, since 
\[
\left|\alpha\left(n\right)\left(B-B^{\gamma\left(n\right)}\right)\right|\leq\alpha\left(n\right)\gamma\left(n\right)=\gamma\left(n\right)^{1/2}
\]
and for $n=2,3\ldots,$ 
\begin{eqnarray*}
\gamma\left(n\right)^{1/2} & = & \beta\left(n\right)^{-1/2}=n^{-1}\beta\left(n-1\right)^{-3}\\
 & \leq & 2^{-1}\beta\left(n-1\right)^{-1/2}=2^{-1}\gamma\left(n-1\right)^{1/2}.
\end{eqnarray*}
Hence the series 
\[
\sum_{n=2}^{\infty}\alpha\left(n\right)\left(B-B^{\gamma\left(n\right)}\right)
\]
is uniformly convergent to a bounded, continuous process, adapted
to the natural filtration of $B.$ We will use facts proved in Appendix
as well as \cite[Theorem 1]{LM2011}, stating that for any continuous
semimartingale $X$ 
\[
\lim_{c\downarrow0}c\cdot TV^{c}\left(X,1\right)=\left\langle X\right\rangle _{1}
\]
(where $TV^{c}\left(X,T\right)$ was defined in Remark \ref{Skoro}),
from which follows that 
\begin{equation}
\lim_{c\downarrow0}c\cdot TV^{c}\left(B,1\right)=1.\label{eq:TVlimit}
\end{equation}
We will also use the Gaussian concentration of $TV^{c}\left(B,T\right)$
(see \cite[Remark 6]{BL}), from which follows that for $c\in\left(0;1\right)$
and $k=1,2,\ldots,$ 
\begin{equation}
\mathbb{E}TV^{c}\left(B,1\right)^{k}\leq C_{k}c^{-k},\label{eq:TVconc}
\end{equation}
where $C_{k}$ is a constant depending on $k$ only.

We have \begin{fact} \label{counterexample1}The sequence of integrals
\[
\int_{0}^{1}Y_{-}\mathrm{d}B^{\gamma\left(n\right)}
\]
diverges. \end{fact}

\begin{proof} Let us fix $n=2,3,4,\ldots$ and split $\int_{0}^{1}Y_{-}\mathrm{d}B^{\gamma\left(n\right)}$
into two summands, $\int_{0}^{1}Y_{-}\mathrm{d}B^{\gamma\left(n\right)}=I+II,$
where 
\[
I=\sum_{m=2}^{n-1}\alpha\left(m\right)\int_{0}^{1}\left(B-B^{\gamma\left(m\right)}\right)\mathrm{d}B^{\gamma\left(n\right)}
\]
and 
\[
II=\int_{0}^{1}\left\{ \alpha\left(n\right)\left(B-B^{\gamma\left(n\right)}\right)+\sum_{m=n+1}^{\infty}\alpha\left(m\right)\left(B-B^{\gamma\left(m\right)}\right)\right\} \mathrm{d}B^{\gamma\left(n\right)}.
\]

Firstly, we consider the second summand, $II.$ Let us notice that
for $m\geq3,$ $\gamma\left(m\right)^{1/2}\leq3^{-1}\gamma\left(m-1\right)^{1/2}$
which implies 
\begin{eqnarray*}
\left|\sum_{m=n+1}^{\infty}\alpha\left(m\right)\left(B-B^{\gamma\left(m\right)}\right)\right| & \leq & \sum_{m=n+1}^{\infty}\alpha\left(m\right)\gamma\left(m\right)=\sum_{m=n+1}^{\infty}\gamma\left(m\right)^{1/2}\\
 & \leq & \gamma\left(n\right)^{1/2}\sum_{l=1}^{\infty}3^{-l}=\frac{1}{2}\gamma\left(n\right)^{1/2}.
\end{eqnarray*}
Hence 
\[
\left|\int_{0}^{1}\sum_{m=n+1}^{\infty}\alpha\left(m\right)\left(B-B^{\gamma\left(m\right)}\right)\mathrm{d}B^{\gamma\left(n\right)}\right|\leq\frac{1}{2}\gamma\left(n\right)^{1/2}\int_{0}^{1}\left|\mathrm{d}B^{\gamma\left(n\right)}\right|=\frac{1}{2}\gamma\left(n\right)^{1/2}\cdot TV\left(B^{\gamma\left(n\right)},1\right).
\]
By the equality (\ref{eq:krejci}) (see the Appendix), 
\[
\alpha\left(n\right)\int_{0}^{1}\left(B-B^{\gamma\left(n\right)}\right)\mathrm{d}B^{\gamma\left(n\right)}=\gamma\left(n\right)^{1/2}TV\left(B^{\gamma\left(n\right)},1\right)
\]
and by two last estimates we get 
\begin{equation}
II\geq\frac{1}{2}\gamma\left(n\right)^{1/2}TV\left(B^{\gamma\left(n\right)},1\right)\geq\frac{1}{2}\gamma\left(n\right)^{1/2}TV^{2\gamma\left(n\right)}\left(B,1\right),\label{eq:estI}
\end{equation}
where the last estimate follows from $TV\left(B^{\gamma\left(n\right)},1\right)\geq TV^{2\gamma\left(n\right)}\left(B,1\right)$
(see (\ref{eq:TVestim}) in the Appendix).

Now let us consider the first summand, $I.$ For $m=2,\ldots,n-1,$
using integration by parts we calculate 
\begin{eqnarray*}
\int_{0}^{1}\left(B-B^{\gamma\left(m\right)}\right)\mathrm{d}B^{\gamma\left(n\right)} & = & \int_{0}^{1}B\mathrm{d}B^{\gamma\left(n\right)}-\int_{0}^{1}B^{\gamma\left(m\right)}\mathrm{d}B^{\gamma\left(n\right)}\\
 & = & \left(B_{1}-B_{1}^{\gamma\left(m\right)}\right)B_{1}^{\gamma\left(n\right)}+\int_{0}^{1}B^{\gamma\left(n\right)}\mathrm{d}B^{\gamma\left(m\right)}-\int_{0}^{1}B^{\gamma\left(n\right)}\mathrm{d}B.
\end{eqnarray*}
By this, the inequality $\left(a+b+c\right)^{2}\leq3\left(a^{2}+b^{2}+c^{2}\right)$
and the Itô isometry we estimate 
\begin{eqnarray}
\mathbb{E}\left(\int_{0}^{1}\left(B-B^{\gamma\left(m\right)}\right)\mathrm{d}B^{\gamma\left(n\right)}\right)^{2} & \leq & 3\gamma\left(m\right)^{2}\mathbb{E}\left(B_{1}^{\gamma\left(n\right)}\right)^{2}\nonumber \\
 &  & +3\mathbb{E}\left\{ \sup_{0\leq s\leq1}\left(B_{s}^{\gamma\left(n\right)}\right)^{2}TV\left(B^{\gamma\left(m\right)},1\right)^{2}\right\} \nonumber \\
 &  & +3\int_{0}^{1}\mathbb{E}\left(B_{s}^{\gamma\left(n\right)}\right)^{2}\mathrm{d}s.\label{eq:estsq}
\end{eqnarray}
Further, from $a^{2}b^{2}\leq\frac{1}{2}a^{4}+\frac{1}{2}b^{4}$ and
then $\left|B_{s}^{\gamma\left(n\right)}\right|\leq\left|B_{s}\right|+\gamma\left(n\right),$
$TV\left(B^{\gamma\left(m\right)},1\right)\leq TV^{2\gamma\left(m\right)}\left(B,1\right)+2\gamma\left(m\right)$
(this follows from the estimate (\ref{eq:TVestim})) and $\left(a+b\right)^{4}\leq8\left(a^{4}+b^{4}\right),$
\begin{eqnarray*}
\mathbb{E}\left\{ \sup_{0\leq s\leq1}\left(B_{s}^{\gamma\left(n\right)}\right)^{2}TV\left(B^{\gamma\left(m\right)},1\right)^{2}\right\}  & \leq & \frac{1}{2}\mathbb{E}\sup_{0\leq s\leq1}\left(B_{s}^{\gamma\left(n\right)}\right)^{4}+\frac{1}{2}\mathbb{E}TV\left(B^{\gamma\left(m\right)},1\right)^{4}\\
 & \leq & \frac{1}{2}8\mathbb{E}\sup_{0\leq s\leq1}\left(B_{s}^{4}+\gamma\left(n\right)^{4}\right)+\frac{1}{2}8\mathbb{E}\left(TV^{2\gamma\left(m\right)}\left(B,1\right)^{4}+2^{4}\gamma\left(n\right)^{4}\right)\\
 & \leq & 4\mathbb{E}\sup_{0\leq s\leq1}B_{s}^{4}+4\mathbb{E}\sup_{0\leq s\leq1}TV^{2\gamma\left(m\right)}\left(B,1\right)^{4}+1.
\end{eqnarray*}
Similarly, by$\left|B_{s}^{\gamma\left(n\right)}\right|\leq\left|B_{s}\right|+\gamma\left(n\right)$
and $\left(a+b\right)^{2}\leq2\left(a^{2}+b^{2}\right)$ we calculate
\[
\mathbb{E}\left(B_{1}^{\gamma\left(n\right)}\right)^{2}\leq2\mathbb{E}\left(B_{1}^{2}+\gamma\left(n\right)^{2}\right)\leq3
\]
and 
\[
\int_{0}^{1}\mathbb{E}\left(B_{s}^{\gamma\left(n\right)}\right)^{2}\mathrm{d}s\leq3.
\]

Hence, by (\ref{eq:estsq}) and last three estimates, 
\begin{eqnarray}
\mathbb{E}\left(\sum_{m=2}^{n-1}\alpha\left(m\right)\int_{0}^{1}\left(B-B^{\gamma\left(m\right)}\right)\mathrm{d}B^{\gamma\left(n\right)}\right)^{2} & \leq & n\sum_{m=2}^{n-1}\alpha\left(m\right)^{2}\mathbb{E}\left(\int_{0}^{1}\left(B-B^{\gamma\left(m\right)}\right)\mathrm{d}B^{\gamma\left(n\right)}\right)^{2}\nonumber \\
 & \leq & n\sum_{m=2}^{n-1}\alpha\left(m\right)^{2}3\left(3\gamma\left(m\right)^{2}+4\mathbb{E}\sup_{0\leq s\leq1}B_{s}^{4}+4\mathbb{E}TV^{2\gamma\left(m\right)}\left(B,1\right)^{4}+4\right)\nonumber \\
 & \leq & n^{2}\alpha\left(n-1\right)^{2}3\left(7+4\mathbb{E}\sup_{0\leq s\leq1}B_{s}^{4}+4\mathbb{E}TV^{2\gamma\left(n-1\right)}\left(B,1\right)^{4}\right).\label{eq:estsq1}
\end{eqnarray}
By the Gaussian concentration properties of $\sup_{0\leq s\leq1}B_{s}$
and $TV^{2\gamma\left(n-1\right)}\left(B,1\right)$ (estimate (\ref{eq:TVconc})),
there exists universal constants $\tilde{C},$ $C$ such that 
\[
\mathbb{E}TV^{2\gamma\left(n-1\right)}\left(B,1\right)^{4}\leq\tilde{C}\gamma\left(n-1\right)^{-4}
\]
and 
\begin{equation}
3\left(7+4\mathbb{E}\sup_{0\leq s\leq1}B_{s}^{4}+4\mathbb{E}TV^{2\gamma\left(n-1\right)}\left(B,1\right)^{4}\right)\leq C\gamma\left(n-1\right)^{-4}=C\beta\left(n-1\right)^{4}.\label{eq:estsq2}
\end{equation}
By (\ref{eq:estsq1}) and (\ref{eq:estsq2}), 
\begin{equation}
\mathbb{E}\left(\sum_{m=2}^{n-1}\alpha\left(m\right)\int_{0}^{1}\left(B-B^{\gamma\left(m\right)}\right)\mathrm{d}B^{\gamma\left(n\right)}\right)^{2}\leq n^{2}\alpha\left(n-1\right)^{2}C\beta\left(n-1\right)^{4}=Cn^{2}\beta\left(n-1\right)^{5}.\label{eq:estsq3}
\end{equation}

Now, by (\ref{eq:estsq3}) and the Chebyshev inequality we get 
\[
\mathbb{P}\left(\left|I\right|\geq\sqrt{3C}n\beta\left(n-1\right)^{5/2}\right)\leq\frac{1}{3}.
\]
Thus, for the set $A_{n}:=\left\{ \left|I\right|\leq\sqrt{3C}n\beta\left(n-1\right)^{5/2}\right\} $
we have $\mathbb{P}\left(A_{n}\right)\geq2/3,$ and by (\ref{eq:estI})
on $A_{n}$ we have 
\begin{eqnarray*}
\int_{0}^{1}Y_{-}\mathrm{d}B^{\gamma\left(n\right)} & = & I+II\geq\frac{1}{2}\gamma\left(n\right)^{1/2}TV\left(B^{\gamma\left(n\right)},1\right)-\sqrt{2C}n\beta\left(n-1\right)^{5/2}\\
 & \geq & \frac{1}{2}\gamma\left(n\right)^{-1/2}\gamma\left(n\right)TV^{2\gamma\left(n\right)}\left(B,1\right)-\sqrt{2C}n\beta\left(n-1\right)^{5/2}\\
 & = & \frac{1}{2}\beta\left(n\right)^{1/2}\gamma\left(n\right)TV^{2\gamma\left(n\right)}\left(B,1\right)-\sqrt{2C}n\beta\left(n-1\right)^{5/2}.
\end{eqnarray*}
Let us choose such $N$ that for any $n\geq N,$ 
\[
\mathbb{P}\left(\gamma\left(n\right)TV^{2\gamma\left(n\right)}\left(B,1\right)\geq\frac{1}{4}\right)\geq\frac{2}{3}
\]
(this is possible by (\ref{eq:TVlimit})). By the definition of $\beta\left(n\right),$
on the set $A_{n}\bigcap D_{n},$ where 
\[
D_{n}:=\left\{ \gamma\left(n\right)TV^{2\gamma\left(n\right)}\left(B,1\right)\geq\frac{1}{4}\right\} ,
\]
we get 
\[
\frac{1}{2}\beta\left(n\right)^{1/2}\gamma\left(n\right)TV^{2\gamma\left(n\right)}\left(B,1\right)-\sqrt{3C}n\beta\left(n-1\right)^{5/2}\geq\frac{1}{8}n\beta\left(n-1\right)^{3}-\sqrt{3C}n\beta\left(n-1\right)^{5/2}.
\]
Since 
\[
\frac{1}{8}n\beta\left(n-1\right)^{3}-\sqrt{3C}n\beta\left(n-1\right)^{5/2}\rightarrow+\infty
\]
as $n\rightarrow+\infty$ and 
\[
\mathbb{P}\left(A_{n}\bigcap D_{n}\right)\geq\frac{1}{3},
\]
we get that the sequence of integrals 
\[
\int_{0}^{1}Y_{-}\mathrm{d}B^{\gamma\left(n\right)}
\]
is divergent. \end{proof}

\begin{rem} From Theorem \ref{thm:leb_stieltjes_stoch} and just
proved Fact \ref{counterexample1} it follows that the bounded, continuous
process 
\[
Y=\sum_{n=2}^{\infty}\alpha\left(n\right)\left(B-B^{\gamma\left(n\right)}\right),
\]
adapted to the natural filtration of $B,$ can not be a semimartingale.\end{rem}

The construction of sequences $Z^{\delta\left(n\right)},$ $\tilde{B}^{\delta\left(n\right)},$
$n=1,2,\ldots$ such that the sequence of integrals $\int_{0}^{1}Z^{\delta\left(n\right)}\mathrm{d}\tilde{B}^{\delta\left(n\right)},$
$n=1,2,\ldots,$ is divergent as $n\uparrow+\infty$ and $Z^{\delta\left(n\right)},$
$\tilde{B}^{\delta\left(n\right)}$ satisfy conditions (1)-(5) for
some semimartingales $Z,$ $\tilde{B}.$ is much easier. We set $\delta\left(n\right)=1/n,$
$Z^{\delta\left(n\right)}=2B^{1/n^{2}}+n\left(B^{1/\left(2n^{2}\right)}-B^{1/n^{2}}\right),$
$\tilde{B}^{\delta\left(n\right)}=B^{1/n^{2}}.$ We easily check that
$Z^{\delta\left(n\right)}$ satisfies (1)-(5) for $Z=2B$ and trivially
$\tilde{B}^{\delta\left(n\right)}$ satisfies (1)-(5) for $\tilde{B}=B.$
Since for any $c>0,$ on the set $B^{c}=B-c,$ $\mathrm{d}B^{c}\geq0$,
and on the set $B^{c}=B+c$, $\mathrm{d}B^{c}\leq0$ (see Lemma \ref{lem:skoroproblem}
in the Appendix), and $-c/2\leq B-B^{c/2}\leq c/2,$ we get $B^{c/2}-B^{c}\geq c/2$
on the set $\mathrm{d}B^{c}>0$ and $B^{c/2}-B^{c}\leq-c/2$ on the
set $\mathrm{d}B^{c}<0.$ Thus 
\begin{align*}
\int_{0}^{1}Z^{\delta\left(n\right)}\mathrm{d}\tilde{B}^{\delta\left(n\right)}-\int_{0}^{1}2B^{1/n^{2}}\mathrm{d}B^{1/n^{2}} & =\int_{0}^{1}n\left(B^{1/\left(2n^{2}\right)}-B^{1/n^{2}}\right)\mathrm{d}B^{1/n^{2}}\\
 & \geq n\frac{1}{2n^{2}}\int_{0}^{1}\left|\mathrm{d}B^{1/n^{2}}\right|=\frac{n}{2}n^{-2}TV\left(B^{1/n^{2}},1\right)\\
 & \geq\frac{n}{2}n^{-2}TV^{1/n^{2}}\left(B,1\right)\cdot
\end{align*}
Now, by the usual Lebesque-Stieltjes integration, $\int_{0}^{1}2B^{1/n^{2}}\mathrm{d}B^{1/n^{2}}=\left(B^{1/n^{2}}\right)^{2},$
and by the just obtained estimate and (\ref{eq:TVlimit}) we see that
\[
\int_{0}^{1}Z^{\delta\left(n\right)}\mathrm{d}\tilde{B}^{\delta\left(n\right)}\rightarrow+\infty.
\]

\section{Appendix}

In this Appendix we will prove estimates used in Section \ref{sec:Counterexamples},
concernig the process $X^{c},$ constructed in Section \ref{sec:Existence}.
Before we proceed, let us recall the definitions of truncated variation,
upward truncated variation and downward truncated variation from Remark
\ref{Skoro}. Let us notice that for $c=0$ we simply get that $TV^{0}$
is the (finite or infinite) total variation and $UTV=UTV^{0}$ and
$DTV=DTV^{0}$ are positive and negative parts of the total variation.
Moreover, we have the Hahn-Jordan decomposition, $TV=UTV+DTV.$ \begin{lem}
For the total variation of the process $X^{c},$ constructed in Section
\ref{sec:Existence}, one has the following estimates 
\begin{equation}
TV^{2c}\left(X,T\right)\leq TV\left(X^{c},T\right)\leq TV^{2c}\left(X,T\right)+2c.\label{eq:TVestim}
\end{equation}

\end{lem} \begin{proof} The lower bound in (\ref{eq:TVestim}) follows
directly from the estimate 
\[
\left|X_{t}^{c}-X_{s}^{c}\right|\geq\max\left\{ \left|X_{t}-X_{s}\right|-2c,0\right\} ,
\]
valid for any $0\leq s<t\leq T,$ which follows directly from inequalities
$\left|X_{s}^{c}-X_{s}\right|\leq c,$ $\left|X_{t}^{c}-X_{t}\right|\leq c$
and the triangle inequality.

To prove the opposite inequality, let us assume that $T_{d}^{2c}X\geq T_{u}^{2c}X$
and denote $M_{k}^{2c}=\sup_{t\in\left[T_{u,k}^{2c};T_{d,k}^{2c}\right)}X_{t},$
$m_{k}^{2c}=\inf_{t\in\left[T_{d,k}^{2c};T_{u,k+1}^{2c}\right)}X_{t},$
$k=0,1,...,$ and consider three possibilities. 
\begin{itemize}
\item $T\in\left[0;T_{u,0}^{2c}\right).$ In this case $TV\left(X^{c},T\right)=UTV\left(X^{c},T\right)=DTV\left(X^{c},T\right)=0.$ 
\item $T\in\left[T_{u,0}^{2c};T_{d,0}^{2c}\right).$ In this case 
\[
UTV\left(X^{c},T\right)=\sup_{t\in\left[T_{u,0}^{2c};T\right]}X_{t}-c-X_{0},\mbox{ }DTV\left(X^{c},T\right)=0,
\]
and 
\[
TV\left(X^{c},T\right)=UTV\left(X^{c},T\right)+DTV\left(X^{c},T\right).
\]
Now, by the definition of $TV^{2c}$ it is not difficult to see that
\[
TV^{2c}\left(X,T\right)\geq\max\left\{ \sup_{t\in\left[T_{u,0}^{2c};T\right]}X_{t}-X_{0}-3c,0\right\} \geq TV\left(X^{c},T\right)-2c.
\]

\item $T\in\left[T_{u,k}^{2c};T_{d,k}^{2c}\right),$ for some $k=1,2,...$
In this case, using monotonicity of $X^{c}$ on the intervals $\left[T_{u,k}^{2c};T_{d,k}^{2c}\right]$
and $\left[T_{d,k}^{2c};T_{u,k+1}^{2c}\right],k=0,1,...,$ and formula
(\ref{eq:defXc}) we calculate 
\begin{eqnarray*}
UTV\left(X^{c},T\right) & = & \left(M_{0}^{2c}-c-X_{0}\right)+\sum_{i=1}^{k-1}\left(M_{i}^{2c}-m_{i-1}^{2c}-2c\right)\\
 &  & +\sup_{t\in\left[T_{u,k}^{2c};T\right]}X_{t}-m_{k-1}^{2c}-2c,
\end{eqnarray*}
\[
DTV\left(X^{c},T\right)=\sum_{i=0}^{k-1}\left(M_{i}^{2c}-m_{i}^{2c}-2c\right)
\]
and 
\[
TV\left(X^{c},T\right)=UTV\left(X^{c},T\right)+DTV\left(X^{c},T\right).
\]
Now it is not difficult to see that 
\begin{eqnarray*}
UTV^{2c}\left(X,T\right) & \geq & \max\left\{ M_{0}^{2c}-X_{0}-3c,0\right\} +\sum_{i=1}^{k-1}\left(M_{i}^{2c}-m_{i-1}^{2c}-2c\right)\\
 &  & +\sup_{t\in\left[T_{u,k}^{2c};T\right]}X_{t}-m_{k-1}^{2c}-2c\geq UTV\left(X^{c},T\right)-2c,
\end{eqnarray*}
\[
DTV^{2c}\left(X,T\right)\geq\sum_{i=0}^{k-1}\left(M_{i}^{2c}-m_{i}^{2c}-2c\right)=DTV\left(X^{c},T\right)
\]
and 
\[
TV^{2c}\left(X,T\right)=UTV^{2c}\left(X,T\right)+DTV^{2c}\left(X,T\right)\geq TV\left(X^{c},T\right)-2c.
\]

\item $s\in\left[T_{d,k};T_{u,k+1}\right),$ for some $k=0,1,2,...$ The
proof follows similarly as in the previous case. 
\end{itemize}
\end{proof}

Now we will prove that the construction of $X^{c}$ in Section \ref{sec:Existence}
is based on a Skorohod map on the interval $\left[-c;c\right].$ Let
us recall the definition of the Skorohod problem on the interval $\left[-c;c\right].$
Let $D[0;+\infty)$ denotes the set of real-valued càdlàg functions
and $BV^{+}[0;+\infty),$ $BV[0;+\infty)$ denote subspaces of $D[0;+\infty)$
consisting of nondecreasing functions and functions of bounded variation,
respectively. We have

\begin{defi} \label{def:Skoro}A pair of functions $\left(\phi,\eta\right)\in D[0;+\infty)\times BV[0;+\infty)$
is said to be a solution of the \emph{Skorohod problem on $\left[-c,c\right]$
for $\psi$} if the following conditions are satisfied: 
\begin{enumerate}
\item for every $t\geq0,$ $\phi^{c}\left(t\right)=\psi\left(t\right)+\eta^{c}\left(t\right)\in\left[-c,c\right];$ 
\item $\eta=\eta_{l}-\eta_{u},$ where $\eta_{l},\eta_{u}\in BV^{+}[0;+\infty)$
and the corresponding measures $\mathrm{d}\eta_{l},$ $\mathrm{d}\eta_{u}$
are carried by $\left\{ t\geq0:\phi(t)=-c\right\} $ and $\left\{ t\geq0:\phi(t)=c\right\} $
respectively. 
\end{enumerate}
\end{defi}

It is possible to prove that for every $c>0$ there exist a unique
solution to the Skorohod problem on $\left[-c;c\right]$ (cf. \cite[Theorem 2.6 and Corollary 2.4]{BKR})
and we will write $\phi^{c}=\Gamma^{c}\left(\psi\right)$ to denote
the associated map, called he \emph{Skorohod map on $\left[-c,c\right].$
}Now we will prove \begin{lem} \label{lem:skoroproblem}The process
$X^{c},$ constructed in Section \ref{sec:Existence} and the Skorohad
map on $\left[-c;c\right]$ are related via the equlity 
\[
X^{c}=X-\Gamma^{c}\left(X\right)
\]
and the mutually singular measures $\mathrm{d}UTV\left(X^{c},\cdot\right)$
and $\mathrm{d}DTV\left(X^{c},\cdot\right)$ are carried by $\left\{ t\geq0:X_{t}-X_{t}^{c}=c\right\} $
and $\left\{ t\geq0:X_{t}-X_{t}^{c}=-c\right\} $ respectively. Thus,
on these sets we have 
\[
\mathrm{d}UTV\left(X^{c},\cdot\right)=\mathrm{d}X^{c},\mbox{ }\mathrm{d}DTV\left(X^{c},\cdot\right)=\mathrm{-d}X^{c}
\]
respectively. \end{lem} \begin{proof} Denote $V=X-X^{c}.$ We have
$V\in\left[-c;c\right],$ i.e. condition 1. in the Definition \ref{def:Skoro}
holds, and to finish the proof it is enough to prove that for the
finite variation process $-X^{c}$ the corresponding measures $\mathrm{d}UTV\left(-X^{c},\cdot\right)=\mathrm{d}DTV\left(X^{c},\cdot\right)$
and $\mathrm{d}DTV\left(-X^{c},\cdot\right)=\mathrm{d}UTV\left(X^{c},\cdot\right)$
are carried by $\left\{ t\geq0:V_{t}=-c\right\} $ and $\left\{ t\geq0:V_{t}=c\right\} $
respectively. Notice that by the formula (\ref{eq:defXc}) the process
$-X^{c}$ is nonincreasing on the intervals $\left[T_{u,k}^{2c};T_{d,k}^{2c}\right)$
and nondecreasing on the intervals $\left[T_{d,k}^{2c};T_{u,k+1}^{2c}\right),k=0,1,2,....$
Thus $\mathrm{d}\left(-X_{s}^{c}\right)=\mathrm{d}DTV\left(X^{c},s\right)=\mathrm{-d}\inf_{t\leq s}X_{t}$
and $\mathrm{d}\left(-X_{s}^{c}\right)=-\mathrm{d}UTV\left(X^{c},s\right)=-\mathrm{d\sup}_{t\leq s}X_{t}$
on the intervals $\left(T_{d,k}^{2c};T_{u,k+1}^{2c}\right)$ and $\left(T_{u,k}^{2c};T_{d,k}^{2c}\right),$
$k=0,1,2,...,$ respectively.

Now, notice that the only points of increase of the measure $\mathrm{d}UTV\left(X^{c},\cdot\right)$
from the intervals $\left(T_{u,k}^{2c};T_{d,k}^{2c}\right),k=0,1,2,...$
are the points where the process $X$ attains new suprema. But in
every such point $s$ we have 
\[
X_{s}^{c}=\sup_{t\in\left[T_{u,k}^{2c};s\right]}X_{t}-c=X_{s}-c
\]
and hence $V_{s}=X_{s}-X_{s}^{c}=c.$ Similar assertion holds for
$\mathrm{d}DTV\left(X^{c},\cdot\right).$

Next, notice that at the point $s=T_{u,0}$ one has $X_{s}^{c}=X_{s}-c\geq X_{0}=X_{s-},$
and since for $T_{u,k+1}^{2c}<+\infty,k=0,1,...,$ one has 
\[
T_{u,k+1}^{2c}=\inf\left\{ s\geq T_{d,k}^{2c}:X_{s}-\inf_{t\in\left[T_{d,k}^{2c};s\right]}X_{t}>2c\right\} ,
\]
then for $s=T_{u,k+1}^{2c}<+\infty,k=0,1,...,$ $\inf_{t\in\left[T_{d,k}^{2c};s\right]}X_{t}=\inf_{t\in\left[T_{d,k}^{2c};s\right)}X_{t}$
and 
\begin{eqnarray*}
X_{s}^{c} & = & X_{s}-c\geq\inf_{t\in\left[T_{d,k};s\right]}X_{t}+c\\
 & = & \inf_{t\in\left[T_{d,k}^{2c};s\right)}X_{t}+c=X_{s-}^{c}.
\end{eqnarray*}
Thus, at the points $s=T_{u,k}^{2c},k=0,1,...$ we have $\mathrm{d}DTV\left(X^{c},\cdot\right)=0,$
$\mathrm{d}UTV\left(X^{c},\cdot\right)\geq0$ and $V_{s}=c.$

In a similar way one proves that the measure $\mathrm{d}DTV\left(X^{c},\cdot\right)$
is carried by $\left\{ t\geq0:V_{t}=-c\right\} .$

The last assertion follows from the fact that $UTV$ and $DTV$ are
positive and negative parts of $\mathrm{d}X^{c}.$ \end{proof} The
direct consequence of Lemma \ref{lem:skoroproblem} is the equality
\begin{equation}
\int_{0}^{T}\left(X-X^{c}\right)\mathrm{d}X^{c}=c\cdot\int_{0}^{T}\left|\mathrm{d}X^{c}\right|=c\cdot TV\left(X^{c},T\right),\label{eq:krejci}
\end{equation}
which holds for any $c,$ $T>0.$

\end{document}